\RequirePackage{fix-cm}
\documentclass[smallcondensed]{svjour3}     
\usepackage{latexsym}
\usepackage{graphicx}
\graphicspath{{figures/}}
\usepackage{amsmath}
\usepackage{amsfonts}
\usepackage{amssymb}
\usepackage{amssymb, float,empheq}
\usepackage{verbatim}
\usepackage{amsmath,amssymb}
\usepackage{cite}
\usepackage{exscale}
\usepackage{color,ulem}
\usepackage[weather,alpine,misc,geometry]{ifsym}
\usepackage[applemac]{inputenc}

\def\text#1{\mbox{#1}}

\newcommand{\x}[1]{}

\numberwithin{lemma}{section}
\numberwithin{theorem}{section}
\numberwithin{remark}{section}
\numberwithin{proposition}{section}
\numberwithin{example}{section}
\numberwithin{corollary}{section}
\numberwithin{definition}{section}
\numberwithin{example}{section}

\def\beq{\begin{equation}}
\def\eeq{\end{equation}}
\def\baq{\begin{eqnarray}}
\def\eaq{\end{eqnarray}}
\def\baqn{\begin{eqnarray*}}
\def\eaqn{\end{eqnarray*}}

\def\Limsup{\mathop{{\rm Lim}\,{\rm sup}}}

\usepackage{cite}

\begin{document}

\title{ Faces and Support Functions for the Values of Maximal Monotone Operators
}

\titlerunning{Faces and Support Functions for the Values of Maximal Monotone Operators}        

\author{Bao Tran Nguyen \and
        Pham  Duy Khanh   \\
       }
\authorrunning{Bao Tran Nguyen \and Pham  Duy Khanh } 
\institute{Bao Tran Nguyen  \at
Universidad de O'Higgins, Rancagua, Chile
              \\
              Quy Nhon University, Quy Nhon, Vietnam \\
              \email{nguyenbaotran31@gmail.com, baotran.nguyen@uoh.cl}           
           \and
           Pham  Duy Khanh (corresponding author) \at
                          Department of Mathematics, HCMC University of Education, Ho Chi Minh, Vietnam \\
		Center for Mathematical Modeling, Universidad de Chile,  Santiago, Chile\\
\email{pdkhanh182@gmail.com; pdkhanh@dim.uchile.cl} 
}

\date{Received: date / Accepted: date}

\maketitle

\begin{abstract}
 Representation formulas for faces and support functions of the values of maximal monotone operators are established in two cases: either the operators are defined on uniformly Banach spaces with uniformly convex duals,  or their domains have nonempty interiors on reflexive real Banach spaces. 
		Faces and support functions are characterized by the limit values of the minimal-norm selections of maximal monotone operators in the first case while in the second case they are represented by the limit values of any selection of maximal monotone operators. These obtained formulas are applied to study the structure of maximal monotone operators: the local unique determination from their minimal-norm selections, the local and global decompositions, and the unique determination on dense subsets of their domains.
\end{abstract}
\section{Introduction}
	Faces and support functions are important tools in representation and analysis of closed convex sets (see \cite[Chapter V]{HL93}). For a closed convex set, a face is the set of points on the given set which  maximizes some (nonzero) linear form while the support function is the signed distance from the origin point to the supporting planes of that set. The face associated with a given direction can be defined via the value of the support function at this direction \cite[Definition~3.1.3, p. 220]{HB19}. 
	Recently, this notion has been defined and studied for the values of maximal monotone operators in \cite[Sect.~3]{HB19}. In this paper, the authors provided some characterizations for the boundary and faces of the values of maximal monotones operators in Hilbert spaces. 
	Their work is motivated by the applications of these characterizations to the stability issues of semi-infinite linear programming problems. 
	
	\medskip
	Motivated by the study of the structure of maximal monotone operators, our paper will investigate the faces and  support functions for the values of maximal monotone operators in reflexive real  Banach spaces. 
	We aim to establish some representation formulas for the faces and support functions in two cases regarding the uniform convexity of the given spaces and theirs duals, and the nonemptiness of the domains of maximal monotone operators. For the first case, we will extend the characterizations of faces associated with directions in \cite[Theorem~3.2]{HB19} from Hilbert spaces to uniformly convex ones with uniformly convex duals.  In comparison with previous work, where the authors used the properties of solutions of differential inclusions governed by maximal monotone operators,
	the proof here is new, simpler and more directed since we only use  some basic properties of the Yosida approximation of maximal monotone operators. We formulate in the context of uniformly convex spaces since our proof strongly depends on the single-valuedness the duality mapping and its inverse, and the strong convergence of the trajectories generated by Yosida approximation. The obtained characterizations and the graphical density of points of subdifferentiability of convex functions allow us to get the representation formulas for support functions in uniformly convex spaces with uniformly convex duals. For the second case, we will work with maximal monotone operators whose domains have nonempty interiors in reflexive Banach spaces. Under the assumption that the domains of operators have nonempty interiors and the local boundedness of maximal monotone operators we could refine the formulas obtained in the first case. We show that the faces and support functions can be represented by the limit values of any selection of maximal monotone operators.   
	
	\medskip
	Characterizations for faces and support functions allow us to investigate the structure of maximal monotone operators. On uniformly convex spaces with uniformly convex duals, we show the local unique determination of maximal monotone operators from their minimal-norm selections, and their local decompositions when their minimal-norm selections are locally bounded. 
	On reflexive Banach spaces, we get some global decompositions of maximal monotone operators when their domains have nonempty interiors. The global decompositions allow us to prove the unique determination of maximal monotone operators on dense subset of their domains.    
		
	\medskip
	The rest of this paper is structured as follows. In Sect. 2, we recall some basic notations of geometry of reflexive real Banach spaces and monotone operator theory. We also collect preliminary results in this section for the reader's convenience. In Sect. 3, representation formulas for faces and support functions are established in uniformly convex spaces with uniformly convex duals. Theses formulas help us to show the local unique determination and to get the local decomposition of a maximal monotone operator provided that its minimal-norm selection is locally bounded. In Sect. 4, we will work with maximal monotone operators whose domains have nonempty interiors in reflexive real Banach spaces. Under our assumptions, we could refine the formulas for faces and support functions obtained in Sect. 3. The refined formulas allow us  
	to find some global decompositions of maximal monotone operators and to show their unique determination on dense subsets of their domains.    
\section{Basic Definitions and Preliminaries}
	Let $X$ be a real reflexive Banach space with norm $\|\cdot\|$ and $X^*$ its continuous dual.  The value of a functional $x^*\in X^*$
	at $x\in X$ is denoted by $\langle x^*,x\rangle$. The open unit balls on $X$ and $X^*$ are denoted, respectively,  by 
	$\mathbb{B}$ and $\mathbb{B}^*$. For $x\in X$ and $r>0$, the open ball centered at $x$ with radius $r$ is denoted by $B(x;r)$. 
	We use the symbol $\lim$ or  $\rightarrow$ to indicate the strong convergence in $X$, and  $\rightharpoonup$ for weak convergence in $X$ and $X^*$.
	Denote on $X$ the set-valued mapping $J:X\rightrightarrows X^*$
	$$
	J(x):=\{x^*\in X^*: \langle x^*,x\rangle=\|x\|^2=\|x^*\|^2\}, \quad \forall x\in X.
	$$
	The mapping $J$ is called the \textit{duality mapping} of the space $X$. The inverse mapping $J^{-1}:X^*\rightrightarrows X$ defined by $J^{-1}(x^*):=\{x\in X: x^*\in J(x)\}$ also satisfies
	$$
	J^{-1}(x^*)=\{x\in X: \langle x^*,x\rangle=\|x\|^2=\|x^*\|^2\}.
	$$
	Since $X$ is reflexive, $X=X^{**}$ and so $J^{-1}$ is the duality mapping of $X^*$.
	The properties of $J$ are closely related to the geometry of the spaces $X$ and $X^*$. Recall that  a real Banach space  is \textit{uniformly convex} iff for every $0<\varepsilon<2$ there is some $\delta>0$ so that for any two vectors with  $\|x\|=1$ and $\|y\|=1$, the condition $\|x-y\|\geq\varepsilon$ implies that $\|x+y\|\leq 2(1-\delta)$. Clearly, uniformly convex Banach space is also \textit{strictly convex}, i.e., for any two distinct  vectors with  $\|x\|=1$ and $\|y\|=1$ we have $\|x+y\|<2$. 
	Moreover, the Milman--Pettis theorem states that every uniformly convex Banach space is reflexive, while the converse is not true. 
	\begin{proposition} \label{DualityMappping} {\rm (see \cite[Theorem~1.2]{Barbu10})}
		If the dual space $X^*$ is strictly convex,
		then the duality mapping $J:X\rightarrow X^*$  is single-valued and demicontinuous.   
		If the space $X^*$ is uniformly convex, then $J$ is uniformly continuous on every bounded subset of $X$.
	\end{proposition}
	
	The \textit{effective domain} $\operatorname{dom} f$ of an extended real-valued function $f:X\rightarrow\overline{\mathbb{R}}:=\mathbb{R}\cup\{+\infty\}$ is the set of points $x$ where $f(x)\in\mathbb{R}$. The function $f$ is \textit{proper} if  $\text{dom} f\ne\emptyset$.
	It is \textit{lower semicontinuous} if
	$$
	f(x)\leq\liminf_{y\rightarrow x}f(y)
	$$
	for all $x\in X$. The epigraph of $f$ is defined by
	$$
	\operatorname{epi} f:=\{(x,r):x\in\text{dom} f, r\geq f(x)\}.
	$$
	Suppose now that $f$ is a convex lower semicontinuous function, i.e. $\operatorname{epi} f$ is convex and closed in $X\times\mathbb{R}$. 
	A functional $x^*\in X^*$ is said to be a \textit{subgradient} of $f$ at $x\in X$, if $f(x)$ is finite and 
	$$
	f(y)-f(x)\geq \langle x^*,y-x\rangle, \quad \forall y\in X.
	$$
	The collection of all subgradients of $f$ at $x$ is called the \textit{subdifferential} of $f$ at $x$, that is,
	$$
	\partial f(x):=\{x^*\in X^*: f(y)-f(x)\geq \langle x^*,y-x\rangle,\;\forall y\in X\}.
	$$
	The function $f$ is said to be \textit{subdifferentiable} at $x$ if $f(x)$ is finite and $\partial f(x)\ne\emptyset$.  Clearly, $\partial f(x)$ is convex and weakly closed in $X^*$. The following result represents the graphical density of points of subdifferentiability of $f$ (see \cite{Borwein82} and \cite{BR65}).
	\begin{proposition}\label{Density}  Let $f$ be a proper lower semicontinuous convex function from $X$ into $\overline{\mathbb{R}}$. Then for any $\bar{x}\in\operatorname{dom} f$ and any $\varepsilon>0$ there exists $x\in X$ such that $\partial f(x)\ne\emptyset$ and 
		$$
		\|x-\bar{x}\|+|f(x)-f(\bar{x})|<\varepsilon.
		$$
	\end{proposition}
	
	Given a nonempty  set $S\subset X$, $\operatorname{int}S$ is the \textit{interior} of $S$, $\overline{S}$ is the \textit{closure} of $S$ and bd($S$) is the \textit{boundary} of $S$ with respect to strong topology on $X$. 
	Suppose now that $S$ is nonempty closed and convex. For every $x\in S$, the \textit{tangent cone} and the  \textit{normal cone}  of $S$ at $x$ (see \cite[Section~2.2.4]{BS00} or \cite[Section~4.2]{AubinFrankowska09}) are defined respectively as 
	\begin{equation}\label{NormalTangent}
	T(x;S):=\overline{\bigcup_{t>0}t^{-1}(S-x)},\quad 
	N(x;S):=\{x^*\in X^*: \sup_{y\in S}\langle x^*,y-x\rangle\leq 0\}.
	\end{equation}
	The tangent cone can be expressed in terms of sequences \cite[Proposition~4.2.1]{AubinFrankowska09} as
	\begin{equation}\label{TB}
	T(x;S)=\left\{v\in X: \exists\;\mbox{sequences}\; t_n\downarrow 0, v_n\rightarrow v\;\mbox{with}\;  x+t_nv_n\in S\;\mbox{for all}\;n\in\mathbb{N}\right\}.
	\end{equation}
	By the bipolar theorem \cite[Proposition~2.40]{BS00} we have the following dual relationships
	$$
	T(x;S)=\{v\in X: \sup_{x^*\in N(x;S)}\langle x^*,v\rangle\leq 0\},
	$$
	$$
	N(x;S)=\{x^*\in X^*: \sup_{v\in T(x;S)}\langle x^*,v\rangle\leq 0\}.
	$$
	The function $I_S: X\rightarrow\overline{\mathbb{R}}$ defined by 
	\begin{equation}\label{Indicator}
	I_S(x):=
	\begin{cases}
	0 & \text{if}\; x\in S,\\
	+\infty & \text{otherwise},
	\end{cases}
	\end{equation}
	is called the \textit{indicator function} of $S$ and its dual function $\sigma_S:X^*\rightarrow\overline{\mathbb{R}}$,
	\begin{equation}\label{Support}
	\sigma_S(x^*):=\sup\{\langle x^*,s\rangle: s\in S\}, \quad \forall x^*\in X^*,
	\end{equation}
	is called the \textit{support function} of $S$ (see \cite[p. 79]{Zalinescu02}). 
	
	Similarly, for a nonempty closed and convex set $K\subset X^*$ and $x^*\in K$, we can define the normal cone $N(x^*;K)\subset X$ and the tangent cone $T(x^*;K)\subset X^*$ of $K$ at $x^*$ as \eqref{NormalTangent}. The indicator function $I_K:X^*\rightarrow\mathbb{R}$ and the support function $\sigma_K:X\rightarrow\mathbb{R}$ are also defined similarly as \eqref{Indicator} and \eqref{Support} respectively. Since $X$ is reflexive,  both $\sigma_S$ and $\sigma_K$ are  lower semicontinuous and convex. 
	
	\medskip
	For the set-valued operator $A:X\rightrightarrows X^*$, the \textit{domain} of $A$ is $D(A):=\{x\in X: Ax\ne\emptyset\}$ and $G(A): =\{(x,x^*)\in X\times X^*: x^*\in Ax\}$ is the \textit{graph} of $A$. 
	Recall that $A$ is \textit{monotone},  iff for all $(x,x^*),(y,y^*)\in G(A)$, one has $\langle x^*-y^*, x-y\rangle\geq 0$, and \textit{maximally monotone} iff $A$ is monotone and $A$ has no proper monotone extension (in the sense of graph inclusion). The duality mapping, the subdifferential of a lower semicontinuous proper convex function, the normal cone to a closed convex set are  examples of maximal monotone operators (see \cite[Theorem~A]{Rockafellar70}). 
	The maximal monotone operator $A$ has closed convex values and is \textit{demiclosed} \cite[Proposition~2.1]{Barbu10}, i.e., $A$ satisfies
	$$
	\left[x_n^*\in Ax_n (\forall n\in\mathbb{N}), x_n^*\rightarrow x^*, x_n\rightharpoonup x\right]\Longrightarrow\left[x^*\in Ax\right],
	$$
	$$
	\left[x_n^*\in Ax_n (\forall n\in\mathbb{N}), x_n^*\rightharpoonup x^*, x_n\rightarrow x\right]\Longrightarrow \left[x^*\in Ax\right].
	$$
	Since $X$ is reflexive, $D(A)$ is \textit{nearly convex} (see \cite[Corollary~3.4]{BY14}), i.e., $\overline{D(A)}$ is convex. Moreover, 
	if $\operatorname{int}D(A)\ne\emptyset$ then $\operatorname{int}D(A)=\operatorname{int}\overline{D(A)}$
	(see \cite[Theorem~27.1 and Theorem 27.3]{Simons08}) and $A$ is locally bounded at every
	$x\in \operatorname{int}D(A)$ (see \cite[Theorem 2.28]{Phelps93} or \cite[Theorem~1]{Rockafellar69}), i.e., there exist $r>0$ and $M>0$ such that $x+r\mathbb{B}\subset D(A)$ and 
	$$
	\sup_{y^*\in Ay}\|y^*\|\leq M, \quad \forall y\in x+r\mathbb{B}.
	$$
	Conversely, if $x\in\overline{D(A)}$ and $A$ is locally bounded at $x$, then $x\in	\operatorname{int}D(A)$ (see \cite[Theorem~1.14]{Phelps97} or \cite[Theorem~3.11.15]{Zalinescu02}). 
	
	\medskip\noindent
	If $X^*$ is uniformly convex then for every $x\in D(A)$, since $Ax$ is nonempty closed and convex, there exists a unique point $x^*_{\text{min}}\in Ax$ such that
	$$
	\|x^*_{\text{min}}\|=\min\{\|x^*\|: x^*\in Ax\}
	$$
	(see \cite[Exercise~3.32]{Brezis11}). Therefore, the single-valued nonlinear operator
	$$
	A^{\circ}:D(A)\subset X\rightarrow X^*, \quad A^{\circ}x:=x^*_{\text{min}}
	$$
	is well-defined; it is called the \textit{minimal-norm selection} of $A$. Let us end this section by recalling some results related to the Yosida approximation of a maximal monotone operator (see \cite[Proposition~2.2]{Barbu10}).
	\begin{proposition} \label{Yosida}
		Suppose that both $X$ and $X^*$ are reflexive and strictly convex, and $A:X\rightrightarrows X^*$ is a maximal monotone operator. For every $x\in X$ and $\lambda>0$, there exists a unique $x_\lambda\in X$ such that
		$$
		0\in J(x_\lambda-x)+\lambda A(x_\lambda).  
		$$
		If $x\in D(A)$ then $x_\lambda\rightarrow x$ and $\lambda^{-1}J(x-x_\lambda)\rightharpoonup A^\circ(x)$ as $\lambda\rightarrow 0$. Moreover, if $X^*$ is uniformly convex, then $\lambda^{-1}J(x-x_\lambda)\rightarrow A^\circ(x)$ as $\lambda\rightarrow 0$ for every $x\in D(A)$.
	\end{proposition}
\section{Representation Formulas in Uniformly Convex Spaces}
In this section, we will establish representation formulas for faces and supports functions in uniformly convex spaces with uniformly convex duals. First, we recall the notion of the face associated with direction of the values of a maximal monotone operator. 
\begin{definition} Let $X$ be a real reflexive Banach space and
	$A:X\rightrightarrows X^*$ a maximal monotone operator.  For $x\in D(A)$ and $v\in X$, we define the set
	$$
	A(x;v):=\{x^*\in Ax: \langle x^*,v\rangle=\sigma_{Ax}(v)\}. 
	$$
	If $v\ne 0$ then $A(x;v)$  is called the face associated with direction $v$ of the value  $Ax$. 
\end{definition} 
\begin{remark} 
	{\rm By the definition of the support function, $x^*\in A(x;v)$ if and only if $v\in N_{Ax}(x^*)$, i.e.,  $x^*\in Ax$ and 
		$$
		\langle y^*-x^*,v\rangle\leq 0, \quad \forall y^*\in Ax. 
		$$
		Moreover,
		$A(x;v)$ is the subdifferential of the convex function $\sigma_{Ax}$ at $v$  (see \cite[Proposition~2.121]{BS00}) and so it is closed and convex.  
	}
\end{remark}
We give two examples of faces associated with directions of values of maximal monotone operators. 
\begin{example} $ $
	\begin{itemize}
		\item Let $X=\mathbb{R}$ and $A:\mathbb{R}\rightrightarrows \mathbb{R}$ a maximal monotone operator  given by 
		$$
		Ax=\begin{cases}
		\{-1\} & \text{if}\quad x<0,\\
		[-1,1] & \text{if} \quad x=0,\\
		\{1\} & \text{if} \quad x>0.
		\end{cases}
		$$
		If $x\ne 0$ then $A(x;v)=Ax$ for all $v\in\mathbb{R}$ (since $Ax$ is singleton). Otherwise,  
		$$
		A(0;v)=\begin{cases}
		\{1\} & \text{if}\quad v>0,\\
		[-1,1] & \text{if} \quad v=0,\\
		\{-1\} & \text{if} \quad v<0.
		\end{cases}
		$$
		\item Let $X$ be a real reflexive Banach space and a maximal monotone operator $Ax=N(x;\overline{\mathbb{B}})$, i.e.,
		$$
		Ax=\begin{cases}
		\mathbb{R}_+J(x) & \text{if}\quad \|x\|=1,\\
		\{0_{X^*}\} & \text{if}\quad \|x\|<1,\\
		\emptyset & \text{if} \quad \|x\|>1.
		\end{cases}
		$$
		If $x\in\mathbb{B}$ then $A(x;v)=Ax=\{0_{X^*}\}$ for all $v\in X$. If $x\in\operatorname{bd}(\mathbb{B})$, i.e., $\|x\|=1$, then 
		$$
		A(x;v)=\begin{cases}
		\displaystyle\bigcup_{\xi\in S}\mathbb{R}_+\xi & \text{if}\quad \displaystyle\sigma_{J(x)}(v)=0,\\
		\{0_{X^*}\} & \text{if} \quad \displaystyle\sigma_{J(x)}(v)<0,\\
		\emptyset & \text{if} \quad \displaystyle\sigma_{J(x)}(v)>0,
		\end{cases}
		$$
		where $S:=\{\xi\in X^*: \xi\in J(x), \langle \xi, v\rangle=0\}$. 
	\end{itemize}
\end{example}
\begin{definition}\label{LocalValues}
	Let $X$ be a real reflexive Banach space and
	$A:X\rightrightarrows X^*$ a maximal monotone operator. For every $x\in D(A)$ and $v\in X$ we define the following sets
	$$
	\begin{array}{rl}
	\displaystyle\Limsup_{w\rightarrow v, t\downarrow 0}A(x+tw):=
	\Big\{x^*\in X^*\, |&\exists\;\mbox{sequences}\; w_n\rightarrow  v, t_n\downarrow 0\;\mbox{and}\;x^*_n\rightarrow x^*\\
	&\mbox {with}\; x_n^*\in A(x+t_nw_n)\; \mbox{for all}\; n\in\mathbb{N}
	\Big\},
	\end{array}
	$$
	$$
	\begin{array}{rl}
	\displaystyle w-\Limsup_{w\rightarrow v, t\downarrow 0}A(x+tw):=
	\Big\{x^*\in X^*\, |&\exists\;\mbox{sequences}\; w_n\rightarrow v, t_n\downarrow 0\;\mbox{and}\;x^*_n\rightharpoonup x^*\\
	&\mbox {with}\; x_n^*\in A(x+t_nw_n)\; \mbox{for all}\; n\in\mathbb{N}
	\Big\}.
	\end{array}
	$$
\end{definition}
\begin{remark}
	{\rm Observe that we have the following inclusions
		\begin{equation}\label{Inclusion}
		\displaystyle\Limsup_{w\rightarrow v, t\downarrow 0}A(x+tw)
		\subset
		\displaystyle w-\Limsup_{w\rightarrow v, t\downarrow 0}A(x+tw)\subset A(x;v). 
		\end{equation}
		The first inclusion follows from Definition~\ref{LocalValues} while the second one is proved similarly as in the proof of \cite[Theorem~3.2]{HB19}. 
	}
\end{remark}
When the operator $A$ is defined on a uniformly convex space with uniformly convex dual we obtain the equalities in \eqref{Inclusion}.
\begin{theorem}\label{Representation1}
	Let $X$ be a real Banach space such that $X$ and $X^*$ are uniformly convex. Let $A:X\rightrightarrows X^*$ be a maximal monotone operator. For every $x\in D(A)$ and $v\in X\setminus\{0\}$ we have 
	\begin{equation}\label{FacesValues}
	A(x;v)=\Limsup_{w\rightarrow v, t\downarrow 0}A(x+tw)=w-\Limsup_{w\rightarrow v, t\downarrow 0}A(x+tw).
	\end{equation}
\end{theorem}
\begin{proof} 
	From \eqref{Inclusion}, to get \eqref{FacesValues} it suffices to 
	check
	\begin{equation}\label{LeftInclusion}
	A(x;v) \subset \underset{w \to v, t \downarrow 0}{\operatorname{Limsup}}\,A(x+tw).
	\end{equation}
	Suppose that $x^*\in A(x;v)$. Let $J$ be the duality mapping on $X$. Since $X$ and $X^*$ are uniformly convex, by  Proposition~\ref{DualityMappping},  both $J$ and $J^{-1}$ are single-valued and continuous with respect to strong topology on $X$ and $X^*$.
	Consider the operator $B:X\rightrightarrows X^*$ given by 
	$$
	By:= Ay-J(v)-x^*,\quad \forall y\in X.
	$$ 
	Clearly, $B$ is  maximal monotone and 
	$\operatorname{dom} B=\operatorname{dom} A$. We first show that $B^{\circ}x=-J(v)$.
	Indeed, since $x^*\in A(x;v)$ we have $x^*\in Ax$ and so
	$$
	-J(v)=x^*-J(v)-x^*\in Ax-J(v)-x^*=Bx.
	$$
	Moreover,  for every $y^*\in Ax$ we have  $\langle x^*-y^*,v\rangle\geq 0$  and 
	\begin{eqnarray*}
		\|-J(v)\| &=&\|v\|^{-1}\langle J(v), v\rangle\\
		&\leq&\|v\|^{-1}\langle J(v)+x^*-y^*, v\rangle\\
		&\leq&\|v\|^{-1}\|y^*-J(v)-x^*\|\|v\|\\
		&=&\|y^*-J(v)-x^*\|.
	\end{eqnarray*}
	Applying Proposition~\ref{Yosida} for the maximal monotone operator $B$ and $x\in\text{dom}B$, we can construct a sequence $\{x_n\}\subset X$ such that 
	\begin{equation}\label{Inclusion1}
	0\in J(x_n-x)+\frac{1}{n}B(x_n),
	\end{equation}
	\begin{equation}\label{Limit1}
	\lim_{n\rightarrow\infty}x_n=x, \quad \text and \quad \lim_{n\rightarrow\infty}[nJ(x-x_n)]=-J(v).
	\end{equation}
	Consider the sequence $\{w_n\}\subset X$ given by $w_n:=n(x_n-x)$ for every $n\in\mathbb{N}$. Then, by \eqref{Inclusion1} and \eqref{Limit1}, we have
	$$
	-J(w_n)+J(v)+x^*\in B(x_n)+J(v)+x^*=A(x_n)=A(x+(1/n)w_n),
	$$
	$$
	\lim_{n\rightarrow\infty}w_n=\lim_{n\rightarrow\infty}[J^{-1}J(n(x_n-x))]=J^{-1}[J(v)]=v,
	$$
	$$
	\lim_{n\rightarrow\infty}[-J(w_n)+J(v)+x^*]=x^*.
	$$
	It follows that $x^*\in  \underset{w \to v, t \downarrow 0}{\operatorname{Limsup}}\,A(x+tw)$ and so \eqref{LeftInclusion} holds. 
	$\hfill\Box$
\end{proof}
\begin{remark}{\rm $ $
		\begin{itemize}
			\item Theorem~\ref{Representation1} generalizes \cite[Theorem~3.2]{HB19} from Hilbert spaces to uniformly convex Banach spaces having uniformly convex duals. Our proof is based on the properties of Yosida approximation of maximal monotone operators and it is simpler than the proof of  \cite[Theorem~3.2]{HB19} where the authors used 
			the properties of solutions of  differential inclusions governed by maximal monotone operators. 
			\item We formulate in the context of uniformly convex spaces since our proof strongly depends on the single-valuedness and the strong continuity of the duality mapping and its inverse.
		\end{itemize}
	} 
\end{remark}

\medskip
The formulas in  \eqref{FacesValues} allow us to characterize the boundaries of the values of maximal monotone operators, by means only of the values at nearby points, which are close enough to the reference point but distinct of it (see \cite[Theorem~3.1]{HB19} in Hilbert setting). 

\begin{corollary} Let $X$ be a real Banach space such that $X$ and $X^*$ are uniformly convex. Let $A:X\rightrightarrows X^*$ be a maximal monotone operator. Then, for every $x\in D(A)$, we have
	\begin{equation}\label{RBoundary}
	\begin{array}{rl}
	\text{\rm bd}(Ax)= \displaystyle\Limsup_{y\rightarrow x, y\ne x}A(y):=	\Big\{x^*\in X^*\, |&\exists\;\mbox{sequences}\; y_n\rightarrow  x\;\mbox{and}\;y^*_n\rightarrow x^* \mbox {with}\\
	&y_n\;\ne x\;\mbox{and}\; y_n^*\in A(y_n) \; \mbox{for all}\; n\in\mathbb{N}
	\Big\}.
	\end{array}
	\end{equation}
\end{corollary}
\begin{proof} Let $x\in D(A)$. Observe that $Ax$ can be represented by its faces in the directions as
	\begin{equation}\label{Boundary}
	\text{\rm bd} (Ax)=\overline{\left(\bigcup_{v\ne 0}A(x;v)\right)}.
	\end{equation}
	Indeed, by the closeness of $\operatorname{bd}(Ax)$,  the set on the right hand side of \eqref{Boundary}  is the subset of $\operatorname{bd}(Ax)$. Hence, we only need to check the reverse inclusion in  \eqref{Boundary}.
	If $x^*\in \operatorname{bd}(Ax)$, since $Ax=\operatorname{dom}I_{Ax}$,  then $x^*\in\overline{\operatorname{dom}I_{Ax}}$ and $x^*\notin\operatorname{int}(\operatorname{dom}I_{Ax})$. Therefore, by the maximal monotonicity of the subdifferential mapping, $\partial I_{Ax}$ is not locally bounded at $x^*$ \cite[Theorem~1.14]{Phelps97}, i.e.,  we can find a sequence $\{x_n^*\}\subset Ax$ such that $x_n^*\rightarrow x^*$ and $N(x_n^*;Ax)=\partial I_{Ax}(x_n^*)$ is not bounded for every $n\in\mathbb{N}$. Let $\{v_n\}\subset X$ be such that $v_n\in N(x_n^*;Ax)\setminus\{0\}$ for every $n\in\mathbb{N}$. Then, $x_n^*\in A(x;v_n)$ with $v_n\ne 0$ for every $n\in\mathbb{N}$ and so $x^*$ belongs the set on the right hand side of \eqref{Boundary}.  
	
	Now we use \eqref{FacesValues} and \eqref{Boundary}  to get \eqref{RBoundary}. We have,
	$$
	\text{\rm bd}(Ax)=\overline{\left(\bigcup_{v\ne 0}\underset{w \to v, t \downarrow 0}{\operatorname{Limsup}}\,A(x+tw)\right)}\subset\overline{\Limsup_{y\rightarrow x, y\ne x}A(y)}= \Limsup_{y\rightarrow x, y\ne x}A(y).
	$$
	Suppose that $\displaystyle x^*\in \Limsup_{y\rightarrow x, y\ne x}A(y)$. Then, there exist sequences  $\{x_n\}\subset X$ and $\{x_n^*\}\subset X^*$ such that $x_n\rightarrow x, x_n^*\rightarrow x^*$ and $x_n\ne x, x_n^*\in Ax_n$ for every $n\in\mathbb{N}$. By the maximal monotonicity of $A$, we have $x^*\in Ax$. We will show that $x^*\in\text{bd}(Ax)$.  Suppose on the contrary that $x^*\in\text{int}(Ax)$. Then, for sufficiently large $n$, we have $x_n^*\in \text{int}(Ax)\subset Ax$ and so $x_n^*\in A(x;x_n-x)$ with $x_n-x\ne 0$. By  \eqref{Boundary}, for sufficiently large $n$, $x_n^*\in \text{bd}(Ax)$ which is a contradiction. 
	$\hfill\Box$ 
\end{proof}

\medskip
Now, we use Theorem~\ref{Representation1} to obtain a representation for the support function of the values of $A$ via its minimal-norm selection $A^{\circ}$. First, we consider a relationship of $A(x;v)$ and $A^{\circ}$ when $A(x;v)$ is nonempty.  
\begin{lemma}\label{MinimalFace}  Suppose that $X$ and $X^*$ are uniformly convex Banach spaces. Let $A:X\rightrightarrows X^*$ be a maximal monotone operator. If $x\in D(A)$ and $v\in X\setminus\{0\}$ such that $A(x;v)\ne\emptyset$ then 
	\begin{equation}\label{Minimal}
	\begin{array}{rl}
	w-\underset{\substack{w\rightarrow v, t\downarrow 0\\ x+tw\in D(A)}}{\operatorname{Limsup}}\,A^{\circ}(x+tw):=	\Big\{x^*\in X^*\, |&\exists\;  \mbox{sequences}\; w_n\rightarrow  v, t_n\downarrow 0\;\mbox{such that}\;\,A^{\circ}(x+t_nw_n)\rightharpoonup x^*\\
	&\quad\mbox{and}\; x+t_nw_n\in D(A) \; \mbox{for all}\; n\in\mathbb{N}
	\Big\}.
	\end{array}
	\end{equation}
	is a nonempty subset of $A(x;v)$.
\end{lemma}
\begin{proof} 
	By Theorem~\ref{Representation1}, the set in  \eqref{Minimal} is a subset of $A(x;v)$ and
	since $A(x;v)\ne\emptyset$  we have
	$$
	w-\Limsup_{w\rightarrow v, t\downarrow 0}A(x+tw)\ne\emptyset.
	$$ 
	Let $x^*$ belong to this set. Then,
	there exist sequences  $t_n\downarrow 0, w_n\rightarrow v, x_n^*\rightharpoonup x^*$ with $x_n^*\in A(x+t_nw_n)$ for every $n\in\mathbb{N}$.  It follows that $x+t_nw_n\in D(A)$ for every $n\in\mathbb{N}$ and 
	$\{x_n^*\}$ is bounded in $X^*$. Since $\|A^{\circ}(x+t_nw_n)\|\leq \|x_n^*\|$ ,  $\{A^{\circ}(x+t_nw_n)\}$ is also bounded in $X^*$. Since $X^*$ is reflexive, the sequence  $\{A^{\circ}(x+t_nw_n)\}$ has a subsequence converging weakly to some $ \bar{x}\in X^*$ and  
	it belongs to the set in \eqref{Minimal}. 
	$\hfill\Box$
\end{proof}
The next example shows that the following inclusion
$$
w-\underset{\substack{w\rightarrow v, t\downarrow 0\\ x+tw\in D(A)}}{\operatorname{Limsup}}\,A^{\circ}(x+tw)\subset A(x;v)
$$
may be strict. 
\begin{example} Let $X$ be a real Hilbert space and $Ax=N(x;\overline{\mathbb{B}})$. Let $x_0, v_0\in H$ be such that 
	$$
	\|x_0\|=\|v_0\|=1 \quad \text{and} \quad \langle x_0,v_0\rangle=0.
	$$
It is clear that $w-\underset{\substack{w\rightarrow v_0, t\downarrow 0\\ x_0+tw\in D(A)}}{\operatorname{Limsup}}\,A^{\circ}(x_0+tw)=\{0\}$ (since $A^{\circ}x=0$ for all $x\in\overline{\mathbb{B}}$) while 
$$
A(x_0;v_0)=Ax_0=\mathbb{R}_+x_0
$$
(since $\langle v_0, x^*\rangle=0$ for all $x^*\in Ax_0$). 
\end{example}
\begin{theorem}\label{SF_Rep1}
	Suppose that $X$ and $X^*$ are uniformly convex Banach spaces. Let $A:X\rightrightarrows X^*$ be a maximal monotone operator. For every $x\in D(A)$ and $v\in X\setminus\{0\}$, we have
	\begin{equation}\label{Rep}
	\sigma_{Ax}(v)=
	\begin{cases}
	\displaystyle\liminf_{\substack{w\rightarrow v, t\downarrow 0\\ x+tw\in D(A)}}\langle A^{\circ}(x+tw),w\rangle & \text{if}\;  v\in T(x;\overline{D(A)}),\\
	+\infty & \text{otherwise}.
	\end{cases}
	\end{equation}
\end{theorem}
\begin{proof} If $v\notin T(x;\overline{D(A)})$ then there exists $x^*\in N(x;\overline{D(A)})$ such that $\langle x^*,v\rangle>0$. Then, for every $y^*\in Ax$ and $t>0$, we have
	$	y^*+tx^*\in Ax+ N(x;\overline{D(A)})=Ax$ by the maximal monotonicity of $A$. It follows that $\sigma_{Ax}(v)\geq \langle y^*,v\rangle+t\langle x^*,v\rangle$. Taking $t\rightarrow+\infty$ in the latter inequality, we get $\sigma_{Ax}(v)=+\infty$. 
	
	Suppose now that $v\in T(x;\overline{D(A)})$. It follows from \eqref{TB} that there exist sequences $t_n\downarrow 0, w_n\rightarrow v$ with $x+t_nw_n\in D(A)$ for all $n\in\mathbb{N}$.
	Let $(t_n)$ and $(w_n)$ be any such sequences. Then, we have
	\begin{equation}\label{I1}
	\sigma_{Ax}(w_n)\leq \langle A^{\circ}(x+t_nw_n), w_n\rangle. 
	\end{equation}
	Indeed, by the monotonicity of $A$, for every $x^*\in Ax$, we have 
	$$
	\langle A^{\circ}(x+t_nw_n)-x^*, w\rangle=t^{-1}\langle A^{\circ}(x+t_nw_n)-x^*, x+t_nw_n-x\rangle\geq 0.
	$$
	Hence, $\langle A^{\circ}(x+t_nw_n), w_n\rangle\geq \langle x^*, w_n\rangle$ for every $x^*\in Ax$ and so \eqref{I1} holds. Taking $n\rightarrow\infty$ in \eqref{I1}, by the lower semicontinuity of $\sigma_{Ax}$, we get
	$$
	\sigma_{Ax}(v)\leq \liminf_{n\rightarrow\infty}\langle A^{\circ}(x+t_nw_n),w_n\rangle.
	$$
	Hence, 
	\begin{equation}\label{LHS}
	\sigma_{Ax}(v)\leq \liminf_{\substack{w\rightarrow v, t\downarrow 0\\ x+tw\in D(A)}}\langle A^{\circ}(x+tw),w\rangle. 
	\end{equation}
	Now we establish the reverse inequality when $\sigma_{Ax}(v)<+\infty$. 
	To do this, we only need to point out the existence of the sequences $t_n\downarrow 0, w_n\rightarrow v$ with $x+t_nw_n\in D(A)$ for every $n\in\mathbb{N}$ such that
	\begin{equation}
	\langle A^{\circ}(x+t_nw_n),w_n\rangle\rightarrow\sigma_{Ax}(v).
	\end{equation}
	Applying Proposition~\ref{Density} for the proper lower semicontinuous convex function $\sigma_{Ax}$ and $v\in\operatorname{dom}\sigma_{Ax}$, we can  find a sequence $(v_n)\subset X$  such that $v_n\rightarrow v, \sigma_{Ax}(v_n)\rightarrow\sigma_{Ax}(v)$
	and $A(x;v_n)=\partial\sigma_{Ax}(v_n)\ne\emptyset$. By Lemma~\ref{MinimalFace}, for every $n\in\mathbb{N}$, there exists sequences  $t_m^n\downarrow 0, w_m^n\rightarrow v_n$ as $m\rightarrow\infty$  with $x+t_m^nw_m^n\in D(A)$ for every $m\in\mathbb{N}$ such that
	$\langle A^{\circ}(x+t_m^nw_m^n),w_m^n\rangle\rightarrow\sigma_{Ax}(v_n)$ as $m\rightarrow\infty$.  For every $n\in\mathbb{N}$, choosing $m$ such that 
	$$
	t_m^n\leq\frac{1}{n}, \|w^n_m-v_n\|\leq \frac{1}{n}, \left|\langle A^{\circ}(x+t_m^nw_m^n),w_m^n\rangle-\sigma_{Ax}(v_n)\right|\leq\frac{1}{n}
	$$
	and setting $t_n:=t_m^n, w_n:=w_m^n$. Then, $t_n\downarrow 0, w_n\rightarrow v$ with $x+t_nw_n\in D(A)$ for every $n\in\mathbb{N}$ and 
	$ \langle A^{\circ}(x+t_nw_n),w_n\rangle\rightarrow\sigma_{Ax}(v)$. 
	Hence, we have the equality in \eqref{LHS}.
	$\hfill\Box$
\end{proof}
\begin{remark} {\rm It follows from \eqref{Rep} that $(x,x^*)\in G(A)$ if and only if $x\in D(A)$ and
		the following inequality 
		\begin{equation}\label{Neighbor}
		\langle x^*-A^{\circ}y, x-y\rangle\geq 0
		\end{equation}
		holds for all $y$ in some neighborhood of $x$. Indeed, by Theorem~\ref{SF_Rep1}, if $v\in T(x;\overline{D(A)})$ then  
		\begin{eqnarray*}
			\sigma_{Ax}(v)&=&\liminf_{\substack{w\rightarrow v, t\downarrow 0\\ x+tw\in D(A)}}\langle A^{\circ}(x+tw),w\rangle\\
			&\geq&\liminf_{\substack{w\rightarrow v, t\downarrow 0\\ x+tw\in D(A)}}\langle x^*,w\rangle\\
			&=&\langle x^*,v\rangle.
		\end{eqnarray*}
		Therefore, 		$\sigma_{Ax}(v)\geq\langle x^*,v\rangle$ for all $v\in X$ and so $(x,x^*)\in G(A)$.
	}  
\end{remark}

\medskip
The formula \eqref{Rep} helps us to establish a local  reconstruction of a maximal monotone operator from its minimal-norm selection.  
\begin{corollary} 	Suppose that $X$ and $X^*$ are uniformly convex Banach spaces.  Let $A_1$ and $A_2$ be maximal monotone operators from $X$ to $X^*$. If there exist $x_0\in D(A_1)\cap D(A_2)$ and $r>0$ such that $D(A_1)\cap B(x_0;r)=D(A_2)\cap B(x_0;r)$ and $A_1^{\circ}=A_2^{\circ}$ on $D(A_1)\cap B(x_0;r)$ then $A_1=A_2$ on $D(A_1)\cap B(x_0;r)$. 
	In particular, if $D(A_1)=D(A_2)$ and $A_1^{\circ}=A_2^{\circ}$ then $A_1=A_2$. 
\end{corollary}
\begin{proof}  
	Let $x_0\in D(A_1)\cap D(A_2)$ and $r>0$ be such that $D(A_1)\cap B(x_0;r)=D(A_2)\cap B(x_0;r)$ and $A_1^{\circ}=A_2^{\circ}$ on $D(A_1)\cap B(x_0;r)$. Let 
	$x\in D(A_1)\cap B(x_0;r)$. By Theorem~\ref{SF_Rep1} and our assumptions, we obtain $\sigma_{A_1x}=\sigma_{A_2x}$.  Hence, we have
	\begin{eqnarray*}
		A_1x=A_1(x;0)&=&\partial\sigma_{A_1x}(0)\\
		&=&\partial\sigma_{A_2x}(0)=A_2(x;0)=A_2x.
	\end{eqnarray*}
	$\hfill\Box$
\end{proof}

\medskip
The next corollary presents a local decomposition of maximal monotone operator provided that its minimal-norm selection is locally bounded.  As a consequence, if the minimal-norm selection of a maximal monotone operator is bounded with some modulus around some interior point of the domain then the whole values of the maximal monotone operator are also bounded with the same modulus around that point.  
\begin{corollary} Let $A$ be a maximal monotone operator from $X$ to $X^*$ and $x\in D(A)$. Suppose that there exist $r>0$ and $\rho>0$  such that
	\begin{equation}\label{eq2}
	\|A^{\circ}y\|\leq \rho, \quad \forall y\in B(x;r)\cap D(A).
	\end{equation}
	Then, for every $y\in B(x;r)\cap D(A)$, we have
	\begin{equation}\label{eq4}
	Ay\subset N(y;\overline{D(A)})+\rho\overline{\mathbb{B}^*}. 
	\end{equation}
	In particular, if $B(x;r)\subset D(A)$ then $Ay\subset \rho \overline{\mathbb{B}^*}$ for every $y\in B(x;r)$.
\end{corollary}
\begin{proof} Let $y\in B(x;r)\cap D(A)$ and $y^*\in Ay$.  We first show that 
	\begin{equation}\label{eq3}
	\langle y^*, z-y\rangle\leq\rho\|z-y\|, \quad \forall z\in \overline{D(A)}.
	\end{equation}
	Indeed, for every $z\in \overline{D(A)}\setminus\{y\}$, $z-y\in T(y;\overline{D(A)})\setminus\{0\}$, and by Theorem~\ref{SF_Rep1} and \eqref{eq2}
	\begin{eqnarray*}
		\langle y^*, z-y\rangle&\leq&\sigma_{Ay}(z-y)\\
		&=&\liminf_{\substack{w\rightarrow z-y, t\downarrow 0\\ y+tw\in D(A)}}\langle A^{\circ}(y+tw),w\rangle\\
		&\leq&\liminf_{\substack{w\rightarrow z-y, t\downarrow 0\\ y+tw\in D(A)}}\|A^{\circ}(y+tw)\|\|w\| \\
		&\leq&\rho\|z-y\|. 
	\end{eqnarray*}
	Hence, \eqref{eq3}  holds and so $y^*\in  N(y;\overline{D(A)})+\rho\overline{\mathbb{B}^*}$. Therefore, \eqref{eq4} is satisfied. 
	
	If $B(x;r)\subset D(A)$ then $N(y;\overline{D(A)})=\{0\}$ for every $y\in B(x;r)$. By \eqref{eq4} we have $Ay\subset\rho\overline{\mathbb{B}^*}$ for all $y$ in this set.
	$\hfill\Box$
\end{proof}
\section{Representation Formulas in Reflexive Spaces}
In this section, we will work with maximal monotone operators having their domains with nonempty interiors in reflexive Banach spaces. Under these assumptions, we could refine the formulas obtained in the previous section. First, we show that the faces  can be represented via the limit values of any selection of maximal monotone operators. 
\begin{lemma}\label{int}
	Let $X$ be a reflexive real Banach space and $A:X\rightrightarrows X^*$ a maximal monotone operator such that $\operatorname{int}(D(A))\ne\emptyset$. Let $D$ be dense subset of $D(A)$ and $\tilde{A}$ be a selection of $A$. For every $x\in D(A), v\in\operatorname{int}(D(A)-x)$, the following set
	\begin{equation}
	\begin{array}{rl}
	w-\underset{\substack{w\rightarrow v, t\downarrow 0\\ x+tw\in D}}{\operatorname{Limsup}}\,\tilde{A}(x+tw):=	\Big\{x^*\in X^*\, |&\exists\;  \mbox{sequences}\; w_n\rightarrow  v, t_n\downarrow 0\;\mbox{such that}\;\,\tilde{A}(x+t_nw_n)\rightharpoonup x^*\\
	&\quad\mbox{and}\; x+t_nw_n\in D \; \mbox{for all}\; n\in\mathbb{N}
	\Big\}.
	\end{array}	
	\end{equation}
	is a nonempty subset of $A(x;v)$. 
\end{lemma}
\begin{proof}  	Let  $x\in D(A)$ and $v\in\operatorname{int}(D(A)-x)$.
	From \eqref{Inclusion}, we have 
	$$ \omega-\underset{w \to v, t \downarrow 0\atop x+tw \in D}{\operatorname{Limsup}}\,\tilde{A}(x+tw) \subset \omega-\underset{w \to v, t \downarrow 0}{\operatorname{Limsup}}\,A(x+tw)\subset A(x;v).$$
	Since $x +v \in \operatorname{int}(D(A))$,  $A$ is locally bounded around $x+v$ , i.e., there exist $r, M> 0$ such that $Ay\ne\emptyset$ and 
	$Ay \subset M\mathbb{B} \, \text{ for all } y \in x+v+4r\mathbb{B}$.  
	According to \cite[Lemma 4.1]{BY13}, there exists $K>0$ such that
	\begin{equation}\label{m.3}
	\emptyset\ne Az \subset K\mathbb{B}, \,\, \forall z \in (x, x+v+2r\mathbb{B}],
	\end{equation}
	where 
	$$
	(x, x+v+2r\mathbb{B}]:=\{\lambda x+ (1-\lambda)z: \lambda\in  ]0,1], z\in x+v+2r\mathbb{B}\}.
	$$
	Picking any sequences $\{t_n\}\subset\mathbb{R}_+$ and $\{w_n\} \subset X$ converging to $0$ and $v$, respectively, and satisfying $x+t_nw_n \in (x, x+v+r\mathbb{B}]$ for all $n \in \mathbb{N}$. 
	Since $D$ is dense on $D(A)$, for each $n \in \mathbb{N}$, we can find $\nu_n$ such that $ x+t_n(w_n+\nu_n) \in D$ and
	\begin{equation}\label{eq}
	\Vert \nu_n\Vert \le \frac{1}{n},\quad w_n+\nu_n \in v + 2r\mathbb{B}.
	\end{equation}
	Combining  \eqref{m.3}  and  \eqref{eq}, we arrive at $w_n+\nu_n\rightarrow v$ as $n\rightarrow+\infty$ and 
	$$ 
	\tilde{A}(x+t_n(w_n+\nu_n))\in  A(x+t_n(w_n+\nu_n)) \subset  K\mathbb{B}, \quad \forall n \in \mathbb{N}. 
	$$
	Since $X^*$ is reflexive, without loss of generality, we can assume that $\tilde{A}(x+t_n(w_n+\nu_n))\rightharpoonup \xi$ as $n \to +\infty$ and so $\xi\in w-\underset{\substack{w\rightarrow v, t\downarrow 0\\ x+tw\in D}}{\operatorname{Limsup}}\,\tilde{A}(x+tw)$.
	$\hfill\Box$
\end{proof}

Second, we use Lemma~\ref{int} to improve the representation formula \eqref{Rep} in Theorem~\ref{SF_Rep1}. 
\begin{theorem}\label{REP2} Let $X$ be a reflexive real Banach space and $A:X\rightrightarrows X^*$ a maximal monotone operator such that $\operatorname{int}(D(A))\ne\emptyset$. Let $D$ be dense subset of $D(A)$ and $\tilde{A}$ be a selection of $A$. For every $x\in D(A)$ and $v\in X\setminus\{0\}$ 
	\begin{equation}
	\sigma_{Ax}(v)=
	\begin{cases}
	\langle \xi,v\rangle & \text{if}\quad  v\in\operatorname{int}\left(T(x;\overline{D(A)})\right),\\
	\displaystyle\liminf_{\substack{w\rightarrow v, t\downarrow 0 \\ x+tw\in D}}\langle \tilde{A}(x+tw),w\rangle & \text{if}\quad  v\in\operatorname{bd}\left(T(x;\overline{D(A)})\right),\\
	+\infty & \text{otherwise},
	\end{cases}
	\end{equation}
	where $\xi$ is any vector in the set $w-\underset{\substack{w\rightarrow v, t\downarrow 0\\ x+tw\in D}}{\operatorname{Limsup}}\,\tilde{A}(x+tw)$. 
\end{theorem}
\begin{proof} Since $D(A)$ has nonempty interior, $\overline{D(A)}$ is convex and $\text{int}(D(A))=\text{int}(\overline{D(A)})$ (see \cite[Theorem~27.1 and Theorem 27.3]{Simons08}). Let $x\in D(A)$ and $v\in X\setminus\{0\}$.  We consider three cases of $v$. 
	
	\noindent\medskip	
	\textit{Case 1.} $\quad v\notin T(x;\overline{D(A)})$ 
	
	Repeating the first part of the proof of Theorem~\ref{SF_Rep1} we get $\sigma_{Ax}(v)=+\infty$.
	
	\noindent\medskip	
	\textit{Case 2.} $\quad v\in\operatorname{int}\left(T(x;\overline{D(A)})\right)$
	
	By \cite[Proposition~4.2.3]{AubinFrankowska09},
	the interior of the tangent cone can be expressed as  
	$$
	\operatorname{int}\left(T(x;\overline{D(A)})\right)=\bigcup_{h>0}\left(\frac{\text{\rm int}(\overline{D(A)})-x}{h}\right)=\bigcup_{h>0}\left(\frac{\text{\rm int}(D(A))-x}{h}\right)
	=\bigcup_{h>0}\frac{\operatorname{int}(D(A)-x)}{h}. 
	$$
	Hence, there exists $h>0$ such that $hv\in \operatorname{int}(D(A)-x)$. Applying Lemma~\ref{int}, the following sets are nonempty and 
	$$
	w-\underset{\substack{w\rightarrow v, t\downarrow 0\\ x+tw\in D}}{\operatorname{Limsup}}\,\tilde{A}(x+tw)=w-\underset{\substack{w\rightarrow hv, t\downarrow 0\\ x+tw\in D}}{\operatorname{Limsup}}\,\tilde{A}(x+tw)\subset A(x;hv)=A(x;v).
	$$
	Then, for every $\xi\in w-\underset{\substack{w\rightarrow v, t\downarrow 0\\ x+tw\in D}}{\operatorname{Limsup}}\,\tilde{A}(x+tw)$, we have $\sigma_{Ax}(v)=\langle\xi,v\rangle$. 
	
	\noindent\medskip
	\textit{Case 3.} $\quad v\in\operatorname{bd}\left(T(x;\overline{D(A)})\right)$
	
	\medskip
	Since $A$ is monotone, for every $w\in X$ and $t>0$ such that $x+tw\in D\subset D(A)$, we have
	$$
	\langle x^*, w\rangle\leq \langle \tilde{A}(x+tw),w\rangle, \quad \forall x^*\in Ax.
	$$
	This yields that
	\begin{equation}\label{lefteq}
	\sigma_{Ax}(v) \le \underset{ w \to v, t \downarrow 0 \atop x +tw \in D}{\liminf} \langle \tilde{A}(x+tw),w \rangle. 
	\end{equation}
	Picking $v_0 \in \operatorname{int}\left(T(x;\overline{D(A)})\right)$.  From Case 2., we have $\sigma_{Ax}(v_0)<+\infty$. 
	Consider the sequence $\{v_n\}$ given by
	$$
	v_n:= \frac{1}{n}v_0+ \frac{n-1}{n}v, \quad \forall  n \in \mathbb{N}.
	$$
	On one hand, since $\sigma_{Ax}$ is convex, we have
	$$
	\sigma_{Ax}(v_n)\leq \frac{1}{n}\sigma_{Ax}(v_0)+\frac{n-1}{n}\sigma_{Ax}(v)
	$$
	Taking the superior limit both sides of the above inequality we get
	$$
	\limsup_{n\rightarrow\infty}\sigma_{Ax}(v_n)\leq \sigma_{Ax}(v). 
	$$
	Then, the lower semicontinuity of $\sigma_{Ax}$ implies that
	\begin{equation}\label{limitsigma}
	\underset{n \to \infty}{\lim}\,\sigma_{Ax}(v_n)= \sigma_{Ax}(v).
	\end{equation}
	On the other hand, the convexity of  $T(x;\overline{D(A)})$ implies that $v_n\in \operatorname{int}\left(T(x;\overline{D(A)})\right)$ for all $n\in\mathbb{N}$. As in the proof of Case 2., for every $n \in \mathbb{N}$
	$$
	\emptyset \ne w-\underset{\substack{w\rightarrow v_n, t\downarrow 0\\ x+tw\in D}}{\operatorname{Limsup}}\,\tilde{A}(x+tw) \subset A(x; v_n).
	$$
	Then, there exist sequences $\{t_n\}\subset\mathbb{R}_+, \{w_n\}\subset X$ such that
	$$ 
	t_n \le \frac{1}{n},\quad \Vert w_n -v_n\Vert \le \frac{1}{n},\quad  x+t_nw_n \in D
	$$
	and 
	$$
	\vert \sigma_{Ax}(v_n)- \langle \tilde{A}(x+t_nw_n), w_n \rangle \vert\le \frac{1}{n}.
	$$  
	for every $n\in\mathbb{N}$ . 
	Clearly, $t_n\downarrow 0, w_n\rightarrow v$ and by \eqref{limitsigma} 
	$$
	\lim_{n\rightarrow\infty} \langle \tilde{A}(x+t_nw_n), w_n \rangle=\sigma_{Ax}(v).
	$$
	Combining this and \eqref{lefteq} we get 
	$$
	\sigma_{Ax}(v)=\underset{ w \to v, t \downarrow 0 \atop x +tw \in D}{\liminf} \langle \tilde{A}(x+tw),w \rangle. 
	$$
	$\hfill\Box$
\end{proof}

Finally, we employ Theorem~\ref{REP2} and  \cite[Proposition 5.1]{BY13} to get the global decompositions for maximal monotone operators.  Our proof follows the technique of \cite[Theorem~5.2]{BY13}. 
\begin{corollary} Let $X$ be a reflexive real Banach space and $A:X\rightrightarrows X^*$ a maximal monotone operator such that $\operatorname{int}(D(A))\ne\emptyset$. Let $D$ be dense subset of $D(A)$ and $\tilde{A}$ be a selection of $A$.  Then, for every $x \in X$, 
	\begin{align}
	\label{m.4}
	Ax & = \overline{\operatorname{co}}\left\{\bigcup_{v \in \operatorname{int}(D(A))- x}w-\underset{w \to v, t \downarrow 0\atop x+tw \in D}{\operatorname{Limsup}}\,\tilde{A}(x+tw)\right\} +\mathrm{N}(x; \overline{D(A)}) \\
	\label{m.5}
	&= \overline{\operatorname{co}}\left\{w-\underset{y \overset{D}{\to} x}{\operatorname{Limsup}}\,\tilde{A}y\right\} +\mathrm{N}(x; \overline{D(A)}).
	\end{align}
\end{corollary}
\begin{proof}
	By the maximal monotonicity of $A$, we have 
	\begin{equation*}
	\begin{split}
	&\overline{\operatorname{co}}\left\{\bigcup_{v \in \operatorname{int}(D(A))- x}w-\underset{w \to v, t \downarrow 0\atop x+tw \in D}{\operatorname{Limsup}}\,\tilde{A}(x+tw)\right\} +\mathrm{N}(x; \overline{D(A)})\\
	\subset \,\,\,&   \overline{\operatorname{co}}\left\{w-\underset{y  \overset{D}{\longrightarrow}  x}{\operatorname{Limsup}}\,\tilde{A}y\right\} +N(x; \overline{D(A)})\\
	\subset \,\,\,&  Ax +N(x; \overline{D(A)})= Ax.
	\end{split}
	\end{equation*}
	Hence, we only need to show that
	\begin{equation}\label{.*}
	Ax \subset \overline{\operatorname{co}}\left\{\bigcup_{v \in \operatorname{int}(D(A))- x}w-\underset{w \to v, t \downarrow 0\atop x+tw \in D}{\operatorname{Limsup}}\,\tilde{A}(x+tw)\right\} +\mathrm{N}(x; \overline{D(A)}).
	\end{equation}
	We set 
	$$ 
	K:=\overline{\operatorname{co}}\left\{\bigcup_{v \in \operatorname{int}(D(A))- x}w-\underset{w \to v, t \downarrow 0\atop x+tw \in D}{\operatorname{Limsup}}\,\tilde{A}(x+tw)\right\}.
	$$
	It is clear that $K \subset Ax.$ When $x \notin D(A)$, both side of \eqref{.*} are empty sets, hence we can assume that $x \in D(A).$ We will use Theorem~\ref{REP2} to show that
	\begin{equation}\label{sigmaEq}
	\sigma_{Ax}(v) \le \sigma_{K}(v),\quad \forall v \in \operatorname{int}(\mathrm{T}_{\overline{D(A)}})(x).
	\end{equation}
	Let $v\in  \operatorname{int}(\mathrm{T}_{\overline{D(A)}})(x)\setminus\{0\}$. By Theorem~\ref{REP2}, there exist $\xi\in w-\underset{\substack{w\rightarrow v, t\downarrow 0\\ x+tw\in D}}{\operatorname{Limsup}}\,\tilde{A}(x+tw)$ such that $\sigma_{Ax}(v)=\langle\xi,v\rangle$.
	On the other hand, by the formula for the interior of tangent cone, there exist $h>0$ such that
	$hv\in \operatorname{int}(D(A)-x)$. This implies that $\xi \in K$ since 
	$$
	w-\underset{\substack{w\rightarrow v, t\downarrow 0\\ x+tw\in D}}{\operatorname{Limsup}}\,\tilde{A}(x+tw)=w-\underset{\substack{w\rightarrow hv, t\downarrow 0\\ x+tw\in D}}{\operatorname{Limsup}}\,\tilde{A}(x+tw).
	$$
	Hence, \eqref{sigmaEq} is satisfied. According to \cite[Proposition 5.1]{BY13}, we get
	$$ Ax \subset \overline{K+\mathrm{N}(x; \overline{D(A)})}.$$
	Now we show that 
	$$ \overline{K+\mathrm{N}(x; \overline{D(A)})} = K+\mathrm{N}(x; \overline{D(A)}).$$
	Let $x^* \in \overline{K+N(x; \overline{D(A)})}.$ There exist $\{s^*_n\} \subset K, \{\nu^*_n\} \subset N(x; \overline{D(A)})$ such that
	$s^*_n + \nu^*_n \to x^*$.
	We will show that both $\{s^*_n\}$ and  $\{\nu^*_n\}$ are bounded. Suppose on the contrary that $\{\nu^*_n\}$ has a subsequence $\{\nu^*_{n_k}\}$ such that  $\|\nu^*_{n_k}\|\rightarrow\infty$. Without loss of generality, we assume that 
	$$
	\frac{\nu^*_{n_k}}{\Vert \nu^*_{n_k}\Vert} \rightharpoonup \xi^* \;\text{and}\; 	\frac{s^*_{n_k}}{\Vert \nu^*_{n_k}\Vert} \rightharpoonup -\xi^*.
	$$
	Let $x_0 \in \operatorname{int}(D(A))$ and $r>0$ be such that $x_0+r\mathbb{B} \subset \operatorname{int}(D(A))$. Then,
	$$ 
	\left\langle\frac{\nu^*_{n_k}}{\|\nu^*_{n_k}\|}, x-x_0\right\rangle \ge r, \quad \forall k\in\mathbb{N}.
	$$
	Taking the limit both sides of the latter inequality, we get $\langle \xi^*, x-x_0 \rangle \ge r$. On the other hand, for some $x^*_0 \in Ax_0$, since $s^*_{n_k} \in K \subset Ax$, we have by the monotonicity of $A$ that
	$$ 
	\left\langle \frac{s^*_{n_k}}{\|\nu^*_{n_k}\|}, x-x_0\right\rangle\ge\frac{\langle x^*_0, x-x_0 \rangle}{\|\nu^*_{n_k}\|},\quad  \forall k\in\mathbb{N}.
	$$
	Again, taking the limit both sides of the latter inequality we also get
	$\langle -\xi^*, x-x_0 \rangle \ge 0$ which is a contradiction.
	Therefore, both $\{\nu^*_n\}$ and $\{s^*_n\}$ are bounded. Since $X$ is reflexive, there exist subsequences $\{\nu^*_{n_k}\}, \{s^*_{n_k}\}$ of $\{\nu^*_n\}$ and $\{s^*_n\}$ respectively such that
	$\nu^*_{n_k} \rightharpoonup \nu^*\in K$ and $s^*_{n_k} \rightharpoonup s^*\in N(x; \overline{D(A)})$. 
	Hence, we have that
	$$ 
	x^* = \nu^* + s^* \in K+ N(x; \overline{D(A)}).
	$$
	$\hfill\Box$
\end{proof}
\begin{remark}
	The formula \eqref{m.5}  has a similar form to the representation formula in \cite[Theorem~5.2]{BY13}. The first term on the right hand-side of the representation in \cite[Theorem~5.2]{BY13}  is the closure convex hull of the limit values of the given maximal monotone operator on a dense subset of its domain while the first term on the right hand-side  of \eqref{m.5} is only represented by the  closure convex hull of the limit values of any selection of that maximal monotone operator. One of the usefulness of this representation is to allow us to prove the unique determination of maximal monotone operators on dense subsets of their domains and characterize the Lipschitz continuity of a convex function. 
\end{remark}

\begin{corollary} Let $X$ be a reflexive real Banach space and $A, B:X\rightrightarrows X^*$ be two maximal monotone operator such that $\operatorname{int}(D(A))=\operatorname{int}(D(B))\ne\emptyset$. If there exists a dense subset $D$ of $D(A)$ such that 
	\begin{equation}\label{Intersection} 
	Ax \cap Bx \ne \emptyset\,\quad \forall x \in D,
	\end{equation}
	then $A=B$.
\end{corollary}
\begin{proof}  Since $\operatorname{int}(D(A))=\operatorname{int}(D(B))$, we have 
	$$ 
	\overline{D(A)}= \overline{\operatorname{int}(D(A))}= \overline{\operatorname{int}(D(B))}=\overline{D(B)}.
	$$
	From \eqref{Intersection}, we can find selections $\tilde{A}$ of $A$ and $\tilde{B}$ of $B$ such that $\tilde{A}=\tilde{B}$ on $D$.   
	Applying Corollary~\ref{REP2}, for every $x\in X$, we have
	\begin{eqnarray*}
		Ax&=& \overline{\operatorname{co}}\left\{w-\underset{y \overset{D}{\to} x}{\operatorname{Limsup}}\,\tilde{A}y\right\} +\mathrm{N}(x; \overline{D(A)})\\
		&=& \overline{\operatorname{co}}\left\{w-\underset{y \overset{D}{\to} x}{\operatorname{Limsup}}\,\tilde{B}y\right\} +\mathrm{N}(x; \overline{D(B)})\\
		&=&Bx.
	\end{eqnarray*}
	$\hfill\Box$
\end{proof}
	We end this section by the following example.
	\begin{example} Let $X$ be a reflexive real Banach space and $f:X\rightarrow\mathbb{R}$ a lower semicontinuous convex function. Suppose that there exist a dense subset $D$ of $X$ and $\ell\geq 0$ such that
		$$
		\partial f(x)\cap\ell\overline{\mathbb{B^*}}\ne\emptyset, \quad \forall x\in D.
		$$
Then, $f$ is $\ell-$Lipschitz continuous on $X$, i.e.,
$$
|f(x)-f(y)|\leq \ell\|x-y\|, \quad \forall x,y\in X.
$$
Indeed, under our assumptions, it follows from \eqref{m.5} that 
$$
\partial f(x)\subset \ell \overline{\mathbb{B}^*}, \quad \forall x\in X,
$$
which implies that $f$ is $\ell-$Lipschitz continuous on $X$. 
		
	\end{example}
\section{Conclusions}
We have provided representation formulas for faces and support functions for the values of maximal monotone operators in reflexive Banach spaces. 
The obtained representation formulas help us to prove the local unique determination of a maximal monotone operator from its minimal-norm selection or on a dense subset of its domain. Some local and global decompositions for maximal monotone operators are also established. Further developments will be devoted to extending our results to arbitrary Banach spaces.  


\begin{thebibliography}{}
\bibitem{AubinFrankowska09} Aubin, J.-P., Frankowska, H.: Set-Valued Analysis, Modern Birkh\"auser Classics. Birkh\"auser, Boston (2009)
		
		\bibitem{Barbu10}  Barbu V.: Nonlinear Differential Equations of Monotone Types in Banach Spaces.
		Springer Monographs in Mathematics. Springer, New York (2010)
		
		\bibitem{BS00} Bonnans J.F.,  Shapiro A.: Perturbation Analysis of Optimization Problems, Springer, New York (2000)
		
		\bibitem{Borwein82} Borwein J.M.: \textit{A note on $\varepsilon$-subgradients and maximal monotonicity}. Pac. J. Math. \textbf{103}, 307--314 (1982)
		
		\bibitem{BY13} Borwein J.M.,  Yao L.: \textit{Structure theory for maximally monotone operators with points of continuity}. J. Optim. Theory Appl. \textbf{157}, 1--24  (2013) 
		
		\bibitem{BY14} Borwein J.M., Yao L.: \textit{Some results on the convexity of the closure of the domain of a maximally monotone operator}. Optim. Lett. \textbf{8}, 237--246 (2014)
		
		\bibitem{Brezis73}  Br\'ezis H.: Op\'erateurs Maximaux Monotones et Semi-Groupes de Contractions dans les Espaces de Hilbert. North-Holland, Amsterdam (1973)
		
		\bibitem{Brezis11} Br\'ezis H.: Functional Analysis, Sobolev Spaces and Partial Differential Equations. Universitext. Springer, New York (2011)
		
		\bibitem{BR65} Br$\o$ndsted A.,  Rockafellar R.T.:  \textit{On the subdifferentiability of convex functions}. Proc. Amer.
		Math. Soc. \textbf{16}, 605--611 (1965)
		
		\bibitem{HB19} Hantoute A., Bao Tran N.: \textit{Boundary of maximal monotone operators values}. Appl. Math. Optim. DOI 10.1007/s00245-018-9498-5
		
		\bibitem{HL93} Hiriart-Urruty J.-B.,  Lemar\'echal C.: Convex Analysis and Minimization Algorithms I. Springer-Verlag, Berlin (1993)
		
		
		\bibitem{Phelps93} Phelps, R.R.: Convex Functions, Monotone Operators and Differentiability. Lecture Notes in Mathematics, vol. 1364, 2nd edn. Springer, Berlin (1993) 
		
		\bibitem{Phelps97}  Phelps, R.R.: \textit{Lectures on maximal monotone operators}. Extracta Math. \textbf{12}, 193--230 (1997)
		
		\bibitem{Simons08} Simons, S.: From Hahn-Banach to Monotonicity. Lecture Notes in Mathematics, vol. 1693, 2nd edn. Springer, New York (2008)
		
		\bibitem{Rockafellar69} Rockafellar, R.T.: \textit{Local boundedness of nonlinear, monotone operators}. Mich. Math. J. \textbf{16}, 397--407 (1969)
		
		\bibitem{Rockafellar70} Rockafellar, R.T.: \textit{On the maximal monotonicity of subdifferential mappings}. Pac. J. Math. \textbf{33}, 209--216 (1970)
		
		\bibitem{Zalinescu02} Z\u{a}linescu, C.: Convex Analysis in General Vector Spaces. World Scientific, Singapore (2002)
\end{thebibliography}
\end{document}